\documentclass[11pt]{article}
\usepackage{amssymb,latexsym,amsmath,amstext,amsthm,fullpage,tabls,graphicx} 
\newcommand{\B}{\ensuremath{\mathcal{B}}}
\newcommand{\C}{\ensuremath{\mathcal{C}}}
\newcommand{\D}{\ensuremath{\mathcal{D}}}

\newtheorem{theorem}{Theorem}

\newtheorem{corollary}[theorem]{Corollary}

\begin{document}

\title{A Problem Concerning Nonincident Points and Blocks in Steiner Triple Systems}
\author{Douglas R.\ Stinson%
\thanks{Research supported by NSERC grant 203114-2011}
\\David R.\ Cheriton School of Computer Science\\
University of Waterloo\\
Waterloo, Ontario N2L 3G1, Canada
}
\date{\today}

\maketitle

\begin{abstract}In this paper, we study the problem of finding the largest possible
set of $s$ points and $s$ blocks in a Steiner triple system of order $v$, such that that none
of the $s$ points lie on any of the $s$ blocks. We prove that $s \leq (2v+5 - \sqrt{24v+25})/2$.
We also show that equality can be attained in this bound for infinitely many values of $v$.
\end{abstract}

\section{Introduction}

This paper is a continuation of \cite{st}, where we studied 
the problem of finding the largest possible
set of $s$ points and $s$ lines in a projective plane of order $q$, such that that none
of the $s$ points lie on any of the $s$ lines. It was shown in \cite{st} that 
$s \leq 1+(q+1)(\sqrt{q}-1)$ and 
equality can be attained in this bound whenever $q$ is an even power of two, by utilising 
certain maximal arcs in the desarguesian plane PG$(2,q)$.

This problem can also be considered in other types of block designs, such
as BIBDs. Suppose $(X, \B)$ is a $(v,k,\lambda)$-BIBD.
For $Y \subseteq X$ and $\C \subseteq \B$, we say that
$(Y, \C)$ is a {\it nonincident} set of points and blocks 
if $y \not\in B$ for every $y \in Y$ and every $B \in \C$.

Maximal arcs have been studied in the setting of BIBDs (see, e.g., \cite{morgan}),
and it might seem plausible that maximal arcs in BIBDs might be of relevance to 
this problem. However, it turns out that things are a bit more complicated.

If we are going to study this problem for BIBDs, then what better place to start than
with Steiner triple systems? A {\it Steiner triple system of order $v$} (or STS$(v)$),
is a pair $(X, \B)$, where 
$X$ is the set of $v$ {\it points} and $\B$ is a set of $b = v(v-1)/6$ {\it blocks},
such that each block contains three points and every pair of points occurs in a unique block.
It is well-known that $v \equiv 1,3 \bmod 6$ is a necessary and sufficient condition for
the existence of an STS$(v)$.

A {\it maximal arc} in an STS$(v)$ consists of a subset $Y$ of  $(v+1)/2$ points such that every block
meets $Y$ in $0$ or $2$ points. When $v \equiv 3,7 \bmod 12$, STS$(v)$ containing maximal arcs can easily be constructed from 
the standard ``doubling construction'' (see, for example, \cite[\S 3.2]{CR}). 
The number of blocks disjoint from $Y$ is
\[ \frac{v(v-1)}{6} - \binom{(v+1)/2}{2} = \frac{v^2-4v+3}{24}.\]
For $v \equiv 3,7 \bmod 12$, $v \geq 19$, it is easy to see
that  $(v^2-4v+3)/24 > (v+1)/2$.
For these values of $v$, this implies that we can find $s = (v+1)/2$ 
points and nonincident blocks in an STS$(v)$ that contains a maximal arc.
However, it turns out that we can do better, and we will show that the optimal
value of $s$ is roughly $v - \sqrt{6v}$, for infinitely many values of $v$.

Define $f(v)$ to be the maximum integer $s$ such that there exists 
a nonincident set of $s$ points and $s$ blocks in some STS$(v)$. Equivalently,
$f(v)$ is the size of the largest square submatrix of zeroes in the incidence
matrix of any STS$(v)$.
We use a simple combinatorial argument to prove the upper bound $f(v) \leq 
(2v+5 - \sqrt{24v+25})/2$.
We also show that this bound is tight
for infinitely many values of $v$. This is done 
by taking $\C$ to be the blocks in a suitably chosen subsystem 
of the STS$(v)$ and letting $Y$ be the 
points not in this subdesign.

\section{Main Results}

\begin{theorem} 
\label{t1}
For any set $Y$ of $s$ points in an STS$(v)$,
the number of blocks disjoint from $Y$ is at most 
\[  \frac{v(v-1)+s^2 - s(2v-1)}{6} .\]
\end{theorem}

\begin{proof}
Suppose that $(X, \B)$ is an STS$(v)$.
Denote $r = (v-1)/2$; then every point occurs in $r$ blocks in the STS$(v)$.
For a subset $Y\subseteq X$ of $s$ points, define $\B_Y = \{B \in \B : B \cap Y \neq \emptyset\}$
and define $\B'_Y = \B \setminus \B_Y$. Furthermore, for every $B \in \B_Y$, 
define $B_Y = B \cap Y$, and then define $\C = \{B_Y: B \in \B_Y\}$. Observe that
$\C$ consists of the nonempty intersections of the blocks in $\B$ with  the set $Y$.
Denote $c = |\C| = |\B_Y|$.

We will study the set system
$(Y,\C)$.  We have the following equations:
\begin{eqnarray*}
\sum_{C \in \C} 1 &=& c\\
\sum_{C \in \C} |C| &=& rs\\
\sum_{C \in \C} \binom{|C|}{2} &=& \binom{s}{2}.
\end{eqnarray*}
 From the above equations, it follows that 
\begin{eqnarray*}
\sum_{C \in \C} |C|^{2} &=& s(s+r-1).
\end{eqnarray*}
Now we compute as follows:
\begin{eqnarray*}
0 &\leq& \sum_{C \in \C} (|C|-2)(|C|-3) \\
&=& s(s+r-1) - 5rs + 6c ,
\end{eqnarray*}
from which it follows that
\begin{eqnarray*} c &\geq& \frac{s(4r+1) - s^2)}{6}\\
&=& \frac{s(2v-1) - s^2)}{6}.
\end{eqnarray*}
Therefore,
\begin{eqnarray*}
|\B'_Y| &=& b-c\\
& \leq & \frac{v(v-1)}{6} -  \frac{s(2v-1) - s^2)}{6} \\
&=& \frac{v(v-1)+s^2 - s(2v-1)}{6}.
\end{eqnarray*}
\end{proof}

\begin{corollary}
If there exists a nonincident set of $s$ points and $t$ blocks in an STS$(v)$, then 
\[ t \leq \frac{v(v-1)+s^2 - s(2v-1)}{6}.\]
\end{corollary}

Before proving our next general result, we look at a small example. 
In Figure \ref{plot-sts.fig}, we graph the functions
$(v(v-1)+s^2 - s(2v-1))/6$ and $s$ for $v=39$ and $s \leq v$. 
The point of intersection is $(26,26)$ and it is then easy to see that
$f(39) \leq 26$.

\begin{figure}
\caption{Nonincident points and lines when $v=39$}
\label{plot-sts.fig}
\begin{center}
\includegraphics[width=90mm]{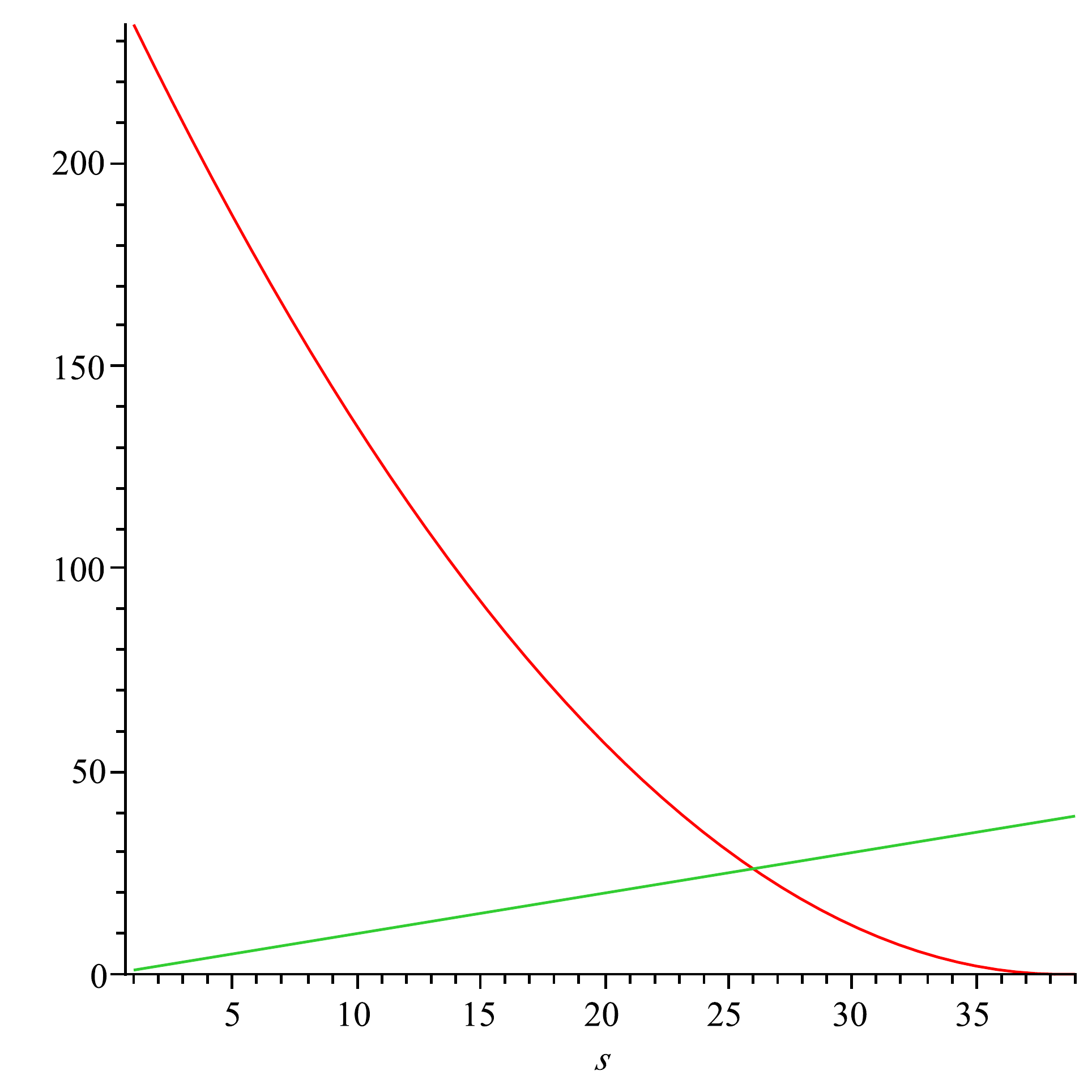}
\end{center}
\end{figure}

In general, it is easy to compute the  point of intersection of these two functions 
as follows:
\begin{eqnarray*} \frac{v(v-1)+s^2 - s(2v-1)}{6} = s &\Leftrightarrow & 
s^2 -(2v+5)s + v^2-v = 0
\\
&\Leftrightarrow & s = \frac{2v+5\pm \sqrt{24v+25}}{2}.
\end{eqnarray*}
Since $s<v$, the point of intersection occurs
when \[s= \frac{2v+5 - \sqrt{24v+25}}{2}.\]

The following result is now straightforward.

\begin{theorem} 
\label{main.thm} For any positive integer $v \equiv 1,3 \bmod{6}$,
it holds that $f(v) \leq \frac{2v+5 - \sqrt{24v+25}}{2}$.
\end{theorem}

\begin{proof} Suppose there is a nonincident set of $s$ points and
$s$ blocks in an STS$(v)$. Theorem \ref{t1} 
implies that $s \leq (v(v-1)+s^2 - s(2v-1))/6$. However, for 
$s > (2v+5 - \sqrt{24v+25})/2$,
we have that $s > (v(v-1)+s^2 - s(2v-1))/6$, just as in the example considered above.
It follows that $s \leq (2v+5 - \sqrt{24v+25})/2$.
\end{proof}

Next, we examine the case of equality in Theorem \ref{t1}. 
This will involve {\it subsystems} of STS$(v)$, which we now define.
Suppose that $(X,\B)$ is an STS$(v)$. 
We say that $(Z,\D)$ is a {\it sub-STS$(w)$} of $(X,\B)$ if
$Z \subset X$, $\D \subset \B$ and
$(Z,\D)$ is an STS$(w)$. It is easy to see that an STS$(v)$ 
containing a sub-STS$(w)$ can exist only if $v \geq 2w+1$.

\begin{corollary}
\label{c3}
Suppose that $(X,\B)$ is an STS$(v)$ and 
suppose we have a set $Y\subset X$ of $s$ points such that 
the number of blocks in $\B$ disjoint from $Y$ is equal to
\[  \frac{v(v-1)+s^2 - s(2v-1)}{6} .\]
Then $(X\backslash Y, \B'_Y)$ is a sub-STS$(v-s)$ of $(X,\B)$,
where $\B'_Y$ denotes the blocks in $\B$ that are disjoint from $Y$.

Conversely, if $(Z,\C)$ is sub-STS$(w)$ of $(X,\B)$, then 
number of blocks in $\B$ disjoint from $X \backslash Z$ is equal to
$(v(v-1)+s^2 - s(2v-1))/6$, where $s = |X \backslash Z| = v-w$.
\end{corollary}

\begin{proof}
 From the proof of Theorem \ref{t1}, it is easy to see that equality holds if and
only if every block in $\B_Y$ meets $Y$ in exactly $2$ or $3$ points. This means
that no block contains two points of $X\backslash Y$, and hence 
$(X\backslash Y, \B'_Y)$ is a sub-STS$(v-s)$ of $(X,\B)$

Conversely, suppose that $(Z,\C)$ is sub-STS$(w)$ of $(X,\B)$. $\C$ consists
of $w(w-1)/6$ blocks, and there are $v-w$ points in $X \backslash Z$.
The blocks in $\C$ are all disjoint from $X \backslash Z$, so it suffices to
verify that
\[ \frac{v(v-1)+(v-w)^2 - (v-w)(2v-1)}{6} = \frac{w(w-1)}{6}.\]
This is an easy computation.
\end{proof}

Our goal is to determine the integers $v$ such that $f(v) = (2v+5 - \sqrt{24v+25})/2$,
i.e., where the bound in Theorem \ref{main.thm} is met with equality.
In view of Corollary,\ref{c3}, this can happen if and only if there is an STS$(v)$ containing
a sub-STS$(v-s)$, where $s= (2v+5 - \sqrt{24v+25})/2$. 
 Therefore, we next determine the integers $v$ such that the following conditions
are satisfied:
\begin{enumerate}
\item $v \equiv 1,3 \bmod 6$
\item 
$s = (2v+5 - \sqrt{24v+25})/2$ is an integer
\item $v-s \equiv 1,3 \bmod 6$, and
\item $v \geq 2(v-s)+1$.
\end{enumerate}

First, condition 2.\ implies  that $24v+25$ is a perfect square, 
say $24v+25 = t^2$. Then we have \[ v= \frac{t^2 - 25}{24}  \quad \text{and} \quad   s = v - \frac{t-5}{2}.\]
Observe that $v$ is an integer only when $t \equiv 1,5 \bmod 6$.

First, suppose $t \equiv 1 \bmod 6$ and write $t = 6u+1$.
It is then easy to see that
\[ v = \frac{3u^2+u -2}{2} \quad \text{and} \quad s = \frac{3u^2-5u +2}{2}.\] 
Now, we consider requirements 1.\ and 3. A straightforward
calculation shows that these conditions are satisfied
if and only if $u \equiv 1,5 \bmod 12$. If we let $u = 12z+1$, then we get
\begin{equation}
\label{case1} v = 216z^2 + 42z + 1 \quad \text{and} \quad s = 216z^2 +6z,
\end{equation} 
while if $u = 12z+5$, we have
\begin{equation}
\label{case2} v = 216z^2 + 186z + 39 \quad \text{and} \quad s = 216z^2 + 150z +26.
\end{equation}

The case $t \equiv 5 \bmod 6$ is handled in a similar way. 
We can write $t = 6u-1$ and then we compute  
\[ v = \frac{3u^2-u -2}{2} \quad \text{and} \quad s = \frac{3u^2-7u +4}{2}.\] 
Here it turns out that requirements 1.\ and 3.\ are satisfied
if and only if $u \equiv 4,8 \bmod 12$. 
If we let $u = 12z+4$, then we get
\begin{equation}
\label{case3}  v = 216z^2 + 138z + 21 \quad \text{and} \quad s = 216z^2 + 102z + 12,
\end{equation} 
while if $u = 12z+8$, we have
\begin{equation}
\label{case4}  v = 216z^2 + 282z + 91 \quad \text{and} \quad s = 216z^2 + 246z +70.
\end{equation}

In all four cases (\ref{case1}), (\ref{case2}), (\ref{case3}) and (\ref{case4}), 
it is easy to see that condition 4.\ is automatically satisfied.
It follows that 
these four cases  are the only situations 
where it is possible to have equality in Theorem \ref{main.thm}. We now show that the
desired designs exist in these cases, by making use of the following well-known
result first proven in \cite{DW}.

\begin{theorem}[Doyen-Wilson Theorem]
\label{DW-thm}
There exists an STS$(v)$ containing a
sub-STS$(w)$ if and only if $v \geq 2w+1$, $v,w \equiv 1,3 \bmod 6$.
\end{theorem}

The above discussion and Theorem \ref{DW-thm} immediately imply
our main existence result.

\begin{theorem} 
\label{main2.thm} Suppose $v \equiv 1,3 \bmod{6}$ is a positive integer. 
Then $f(v) = \frac{2v+5 - \sqrt{24v+25}}{2}$
if and only if 
\[ v \in \{ 216z^2 + 42z + 1, 216z^2 + 186z + 39, 216z^2 + 138z + 21, 216z^2 + 282z + 91\} ,\]
where $z$ is a non-negative integer.
\end{theorem}

The three smallest cases where $f(v)$ attains its optimal value are 
when $v=21$, $s = 12$; $v=39$, $s=26$; and $v=91$, $s=70$.

\end{document}